\documentclass[a4paper]{amsart}
\usepackage{graphicx}        % standard LaTeX graphics tool
\usepackage{multicol}        % used for the two-column index
\usepackage[all]{xy}
\usepackage[mathscr]{eucal}
\usepackage{xspace}
\usepackage{amsmath,amssymb,amsfonts}        % AMS math packages
	\newcommand{\ftn}[3]{ #1 : #2 \rightarrow #3 }
	\newcommand{\setof}[2]{\ensuremath{\left\{ #1 \: : \: #2 \right\}}}
	
	\newcommand{\Z}{\ensuremath{\mathbb{Z}}\xspace}
	\newcommand{\C}{\ensuremath{\mathbb{C}}\xspace}

	\newcommand{\N}{\ensuremath{\mathbb{N}}\xspace}

	\newcommand{\spec}{\textnormal{Spec}}
	\newcommand{\prim}{\textnormal{Prime}}
	\newcommand{\alg}{\textnormal{alg}}
	\newcommand{\tops}{\textnormal{top}}

\theoremstyle{plain}
\newtheorem{theorem}{Theorem}[section]

\newtheorem{lemma}[theorem]{Lemma}
\newtheorem{corollary}[theorem]{Corollary}
\newtheorem{proposition}[theorem]{Proposition}
\newtheorem{conjecture}[theorem]{Conjecture}

\theoremstyle{definition}
\newtheorem{assumption}[theorem]{Standing Assumption}
\newtheorem{definition}[theorem]{Definition}

\numberwithin{equation}{section}
	
\begin{document}

\title{Filtered $K$-theory for graph algebras}
\date{\today}
\author{S\o{}ren Eilers}
\address{Department of Mathe\-matical Sciences, University of Copen\-hagen, Universi\-tets\-park\-en~5, DK-2100 Copen\-hagen, Den\-mark}
\email{eilers@math.ku.dk}
\author{Gunnar Restorff}
\address{Department of Science and Technology, University of the Faroe Islands, N\'{o}at\'{u}n~3, FO-100 T\'{o}rshavn, the Faroe Islands}
\email{gunnarr@setur.fo}
\author{Efren Ruiz}
\address{Department of Mathematics, University of Hawaii, Hilo, 200 W.~Kawili St., Hilo, Hawaii, 96720-4091 USA}
\email{ruize@hawaii.edu}
\author{Adam P.~W.~S{\o}rensen}
\address{Department of Mathematics, University of Oslo, PO BOX 1053 Blindern, N-0316 Oslo, Norway}
\email{apws@math.uio.no}

\begin{abstract}
We introduce filtered algebraic $K$-theory of a ring $R$ relative to a sublattice of ideals.
This is done in such a way that filtered algebraic $K$-theory of a Leavitt path algebra relative to the graded ideals is parallel to the gauge invariant filtered $K$-theory for graph $C^*$-algebras. 
We apply this to verify the Abrams-Tomforde conjecture for a large class of finite graphs.  
\end{abstract}

\maketitle

%%%%%%%%%%%%%%%%%%%%%%%%%
%%%%%%%%%%%%%%%%%%%%%%%%%
\section{Introduction}        
%%%%%%%%%%%%%%%%%%%%%%%%%
%%%%%%%%%%%%%%%%%%%%%%%%%

Since the inception of Leavitt path algebras in \cite{MR2172342, MR2310414} it has been known that there is a strong connection between Leavitt path algebras and graph $C^*$-algebras. 
In particular many results for both graph $C^*$-algebras and Leavitt path algebras have the same hypotheses when framed in terms of the underlying graph and the conclusions about the structure of the algebras are analogous.
For instance, by \cite[Theorem 4.1]{MR1878636} and \cite[Theorem 4.5]{MR2266860} the following are equivalent for a graph $E$.
\begin{enumerate}
 \item $E$ satisfies Condition (K) (no vertex is the base point of exactly one return path). 
 \item $C^*(E)$ has real rank 0.
 \item $L_\C(E)$ is an exchange ring.
\end{enumerate}
That real rank 0 is the analytic analogue of the algebraic property of being an exchange ring is justified in \cite[Theorem 7.2]{MR1639739}. 

One of the most direct connections we could possibly have between Leavitt path algebras and graph $C^*$-algebras would be: If $E,F$ are graphs then
\[
	L_{\C}(E) \cong L_{\C}(F) \iff C^*(E) \cong C^*(F). 
\]
This is called the \emph{isomorphism question} and it is unknown if it is true.
As currently stated the question is very imprecise, while it is clear what is meant by isomorphism of $C^*$-algebras, we could consider isomorphisms of Leavitt path algebras both as rings, algebras, and $*$-algebras.
In the last case the forward implication of the isomorphism question holds.  
In \cite{MR2775826} Abrams and Tomforde take a systematic look at the isomorphism question and many related questions, for instance whether or not the above holds with Morita equivalence in place of isomorphism. 
They provide evidence in favor of a positive answer to the Morita equivalence question and elevate one direction to a conjecture. 

\begin{conjecture}[The Abrams-Tomforde Conjecture]
Let $E$ and $F$ be graphs.
If $L_\C(E)$ is Morita equivalent to $L_\C(F)$, then $C^*(E)$ is (strongly) Morita equivalent to $C^*(F)$.
\end{conjecture}

In \cite{MR3188556} the third named author and Tomforde use ideal related algebraic $K$-theory to verify the Abrams-Tomforde conjecture of large classes of graphs. 
They introduce ideal related algebraic $K$-theory as a Leavitt path algebra analogue for filtered $K$-theory for graph $C^*$-algebras. 
This then allows them to prove the Abrams-Tomforde conjecture for all classes of graphs where the associated $C^*$-algebras are classified by filtered $K$-theory. 

The authors have shown in \cite{arXiv:1604.05439v1} that when classifying graph $C^*$-algebras that do not have real rank 0, it can be useful to replace the full filtered $K$-theory with a version that only looks at gauge invariant ideals.
Motivated by this, we develop a version of ideal related algebraic $K$-theory relative to a sublattice of ideals. 
Our goal is to get an ideal related $K$-theory for Leavitt path algebras that only considers graded ideals, but we try to state our result in greater generality. 
We look at a sublattice $\mathcal{S}$ of ideals in some ring $R$ and consider the spectrum of these ideals, that is the set of $\mathcal{S}$-prime ideals. 
This set is equipped with the Jacobson (or hull-kernel) topology. 
In nice cases there exists a lattice isomorphism from the open sets in the spectrum to the ideals in $\mathcal{S}$. 
Specializing to the case of a Leavitt path algebra $L_{\mathsf{k}}(E)$, we show that the spectrum associated to the graded ideals is homeomorphic to the spectrum of gauge invariant ideals in $C^*(E)$. 
Using this we define filtered algebraic $K$-theory of $L_{\mathsf{k}}(E)$ relative to the graded ideals in complete analogy to the $C^*$-algebra definition. 
We then follow the work of \cite{MR3188556} and establish the Abrams-Tomforde conjecture for all graphs where the $C^*$-algebras are classified by filtered $K$-theory of gauge invariant ideals. 
By \cite{arXiv:1604.05439v1} this includes a large class of finite graphs. 
        
%%%%%%%%%%%%%%%%%%%%%%%%%
%%%%%%%%%%%%%%%%%%%%%%%%%
\section{Preliminaries}
%%%%%%%%%%%%%%%%%%%%%%%%%
%%%%%%%%%%%%%%%%%%%%%%%%%

In this section we set up the notation we will use throughout the paper and we recall the needed definitions. 
We begin with the definitions of graphs, graph $C^*$-algebras and Leavitt path algebras. 

\begin{definition}
A graph $E$ is a quadruple $E = (E^0 , E^1 , r, s)$ where $E^0$ is the set of vertices, $E^1$ is the set of edges, and $r$ and $s$ are maps from $E^1$ to $E^0$ giving the range and source of an edge. 
\end{definition}

\begin{assumption}
Unless explicitly stated otherwise, all graphs are assumed to be countable, i.e., the set of vertices and the set of edges are countable sets.
\end{assumption}

We follow the notation and definition for graph $C^*$-algebras in \cite{MR1670363} and warn the reader that this is not the convention used in the monograph by Raeburn (\cite{MR2135030}). 

\begin{definition} \label{def:graphca}
Let $E = (E^0,E^1,r,s)$ be a graph.
%%%%
% badbox change
%The \emph{graph $C^*$-algebra} $C^*(E)$ is the universal $C^*$-algebra generated by mutually orthogonal projections $\setof{ p_v }{ v \in E^0 }$ and partial isometries $\setof{ s_e }{ e \in E^1 }$ satisfying the relations
The \emph{graph $C^*$-algebra} $C^*(E)$ is defined to be the universal $C^*$-algebra generated by mutually orthogonal projections $\setof{ p_v }{ v \in E^0 }$ and partial isometries $\setof{ s_e }{ e \in E^1 }$ satisfying the relations
\begin{itemize}
	\item $s_e^* s_f = 0$ if $e,f \in E^1$ and $e \neq f$,
	\item $s_e^* s_e = p_{r(e)}$ for all $e \in E^1$,
	\item $s_e s_e^* \leq p_{s(e)}$ for all $e \in E^1$, and,
	\item $p_v = \sum_{e \in s^{-1}(v)} s_e s_e^*$ for all $v \in E^0$ with $0 < |s^{-1}(v)| < \infty$.
\end{itemize}
\end{definition}

We get our definition of Leavitt path algebras from \cite{MR2172342, MR2310414}. 

\begin{definition}
Let $\mathsf{k}$ be a field and let $E$ be a graph. 
The \emph{Leavitt path algebra} $L_{\mathsf{k}}(E)$ is the universal $\mathsf{k}$-algebra generated by pairwise orthogonal idempotents $\{ v \mid v \in E^0\}$ and elements $\{e, e^* \mid e \in E^1\}$ satisfying 
\begin{itemize}
 \item $e^* f = 0$, if $e \neq f$,
 \item $e^* e = r(e)$, 
 \item $s(e) e = e = e r(e)$, 
 \item $e^* s(e) = e^* = r(e) e^*$, and, 
 \item $v = \sum_{e \in s^{-1}(v)} e e^*$, if $s^{-1}(v)$ is finite and nonempty. \label{item:summation_relation}
\end{itemize}
\end{definition}

Recall that graph $C^*$-algebras come with a natural gauge action and that Leavitt path algebras come with a natural grading. 
We now turn to the ideal structure of Leavitt path algebras and graph $C^*$-algebras, where we are particularly interested in graded ideals and gauge invariant ideals.  

\begin{assumption}
Unless explicitly stated otherwise, all ideals in rings are two-sided ideals and all ideals in a $C^{*}$-algebra are closed two-sided ideals.   
\end{assumption}

\begin{definition}
For any ring $R$ we denote by $\mathbb{I}(R)$ the lattice of ideals in $R$.
\end{definition}

As per usual we write $v \geq w$ if there is a path from the vertex $v$ to the vertex $w$.
We call a subset $H \subseteq E^0$ \emph{hereditary} if $v \in H$ and $v \geq w$ imply that $w \in H$, and we say that $H$ is \emph{saturated} if for every $v \in E^0$ with $0 < |s^{-1}(v)| < \infty$ and $r(s^{-1}(v)) \subseteq H$ we have $v \in H$. 
If $H$ is saturated and hereditary we define 
\[
	B_H = \setof{ v \in E^0 \setminus H}{ |s^{-1}(v)|=\infty \text{ and } 0 < |s^{-1}(v) \cap r^{-1}(E^0 \setminus H)| < \infty}.
\]
In other words, $B_H$ consists of infinite emitters that are not in $H$ and emit a non-zero finite number of edges to vertices not in $H$. 
We say that those vertices are \emph{breaking} for $H$.

\begin{definition}[{\cite[Definition~5.4]{MR2363133}}]
An admissible pair $(H, S)$ consists of a saturated hereditary subset $H$ and a subset $S$ of $B_H$.
We put an order on the set of admissible pairs by letting $(H,S) \leq (H',S')$ if and only if $H \subseteq H'$ and $S \subseteq H' \cup S'$. 
This is in fact a lattice order. 
\end{definition}

\begin{theorem}[{\cite[Theorem~3.6]{MR1988256} and \cite[Theorem~5.7]{MR2363133}}] \label{t:admissiblelattice}
Let $E$ be a graph and let $\mathsf{k}$ be a field.
\begin{itemize}
	\item There is a canonical lattice isomorphism from the set of admissible pairs to the set of gauge invariant ideals of $C^*(E)$. We write $I_{(H,S)}^{\tops}$ for the image of an admissible pair.
	\item There is a canonical lattice isomorphism from the set of admissible pairs to the set of graded ideals of $L_{\mathsf{k}}(E)$. We write $I_{(H,S)}^{\alg}$ for the image of an admissible pair.
\end{itemize} 
\end{theorem}

One of the main reasons the sublattice of graded ideals can be used to study the Morita equivalence classes of Leavitt path algebras is that the graded ideals are preserved by (not necessarily graded) ring isomorphisms. 

\begin{lemma} \label{l:gradeidempotents}
Let $E$ be a graph and let $\mathsf{k}$ be a field.
Suppose $I$ is an ideal in $L_{\mathsf{k}}(E)$.
Then $I$ is graded if and only if $I$ is generated by idempotents. 
\end{lemma}
\begin{proof}
Suppose $I$ is graded.  
Then $I = I_{(H,S)}^\alg$ for some admissible pair $(H,S)$.  
By definition (see for instance \cite[Definition~5.5]{MR2363133}) $I_{(H,S)}^\alg$ is generated by $\{ v : v \in H \}$ and 
\[
\{ v - \sum_{ \substack{ s(e) = v \\ r(e) \notin H }} e e^* : v \in S \}.
\] 
Hence $I$ is generated by idempotents.

Suppose instead $I$ is generated by idempotents.
Let $e \in I$ be an idempotent in the generating set $S$ of idempotents for $I$.
By \cite[Theorem~3.4]{MR3310950}, $e$ is equivalent in $M_\infty ( L_\mathsf{k} (G) )$ to a finite sum of the idempotents of the form $v \in E^0$ and $w - \sum_{ i = 1 }^n e_i e_i^*$ where $s(e) = w \in E^0$, $| s^{-1}(w)| = \infty$, and each $e_i$ is an element of $s^{-1}(w)$. 
Then $S_e$ where $e$ is replaced by these new idempotents in the generating set $S$ will generate the ideal $I$.  Thus, $I$ is generated by idempotents in the vertex set and idempotents of the form $v - \sum_{ i = 1 }^n e_i e_i^*$, where $s(e) = v\in E^0$, $|s^{-1}(v) | = \infty$, and each $e_i$ is an element of $s^{-1}(v)$.  
Therefore, $I$ is a graded ideal.  
\end{proof}

%%%%
% badbox change
%Finally we briefly recall from \cite[Section 3]{arXiv:1604.05439v1} the definition of $\prim_\gamma ( C^*(E) )$ and $\mathsf{FK}^{\tops, +} ( \spec_\gamma (C^*(E)) ; C^*(E) )$.
Finally we recall from \cite[Section 3]{arXiv:1604.05439v1} the definition of $\prim_\gamma ( C^*(E) )$ and $\mathsf{FK}^{\tops, +} ( \spec_\gamma (C^*(E)) ; C^*(E) )$.

\begin{definition}
Let $E = (E^0 , E^1 , r , s )$ be a graph. 
Let $\prim_\gamma(C^*(E))$ denote the set of all proper ideals that are prime within the set of proper gauge invariant ideals. 
\end{definition}

We give $\prim_\gamma(C^*(E))$ the Jacobson topology and can then show that $C^*(E)$ has a canonical structure as a $\prim_\gamma(C^*(E))$-algebra.
So when $E$ has finitely many vertices --- or, more generally, $\prim_\gamma(C^*(E))$ is finite --- we can consider the reduced filtered ordered K-theory of $C^*(E)$: $\mathsf{FK}^{\tops, +} ( \spec_\gamma (C^*(E)) ; C^*(E) )$.
Loosely speaking this is the collection of the $K$-groups associated to certain subquotients $I/J$ of gauge invariant ideals $I,J$ in $C^*(E)$ together with certain maps of the associated six-term exact sequences. 

%%%%%%%%%%%%%%%%%%%%%%%%%
%%%%%%%%%%%%%%%%%%%%%%%%%
\section{$\mathcal{S}$-Prime spectrum for a ring} \label{s:primespectrum}
%%%%%%%%%%%%%%%%%%%%%%%%%
%%%%%%%%%%%%%%%%%%%%%%%%%

We will now introduce the Prime-spectrum of a ring relative to a sublattice of ideals. 
Our primary motivation is to look at prime graded ideals in Leavitt path algebras. 

\begin{definition}
Let $R$ be a ring and let $\mathcal{S}$ be a sublattice of $\mathbb{I}(R)$ containing the trivial ideals $\{ 0 \}$ and $R$.
An ideal $P \in \mathcal{S}$ is called \emph{$\mathcal{S}$-prime} if $P \neq R$ and for any ideals $I, J \in \mathcal{S}$,
\begin{align*}
	IJ \subseteq P \quad \implies \quad I \subseteq P \quad \text{or} \quad J \subseteq P.
\end{align*}    
We denote by $\spec_\mathcal{S}(R)$ the set of all $\mathcal{S}$-prime ideals of $R$.
\end{definition}

We note that if $P$ is $\mathcal{S}$-prime and $I,J$ are in $\mathcal{S}$ then $IJ \subseteq I \cap J$ so we have
\[
  I \cap J \subseteq P \quad  \implies \quad I \subseteq P \quad \text{or} \quad J \subseteq P.
\]
We will equip $\spec_\mathcal{S}(R)$ with the Jacobson (or hull-kernel) topology.
For each subset $T \subseteq \spec_\mathcal{S}(R)$ we define the kernel of $T$ as 
\[
  \ker(T) = \bigcap_{\mathfrak{p} \in T} \mathfrak{p}
\]
and the closure of $T$ as
\begin{align} \label{eq:def-closure}
  \overline{T} = \setof{ \mathfrak{p} \in \spec_\mathcal{S}(R) }{ \mathfrak{p} \supseteq \ker(T) }.
\end{align}
Note that if $R$ is a commutative ring and $\mathcal{S} = \mathbb{I} (R)$, then $\spec_\mathcal{S} (R)$ is the spectrum of $R$ with the Zariski topology.

\begin{lemma} \label{l:closure}
Let $R$ be a ring and let $\mathcal{S}$ be a sublattice of $\mathbb{I}(R)$ closed under arbitrary intersections and containing the trivial ideals $\{ 0 \}$ and $R$.
The closure operation defined in (\ref{eq:def-closure}) satisfies the Kuratowski closure axioms, that is
\begin{enumerate}
 \item $\overline{\emptyset} = \emptyset$, \label{ei:closure-empty} 
 \item $T \subseteq \overline{T}$, for all $T \subseteq \spec_\mathcal{S}(R)$, \label{ei:closure-subset}
 \item $\overline{T} = \overline{\overline{T}}$, for all $T \subseteq \spec_\mathcal{S}(R)$, and, \label{ei:closure-idempotent}
 \item $\overline{T_1 \cup T_2} = \overline{T_1} \cup \overline{T_2}$, for all $T_1, T_2 \subseteq \spec_\mathcal{S}(R)$. \label{ei:closure-union}
\end{enumerate}
\end{lemma}
\begin{proof}
Once we recall that by definition $\ker(\emptyset) = R$ it is clear that \ref{ei:closure-empty}. holds and since we have $\mathfrak{p} \supseteq \ker(T)$ for all $\mathfrak{p} \in T$, \ref{ei:closure-subset}. also holds. 
For \ref{ei:closure-idempotent}. we observe that $\ker(\overline{T}) = \ker(T)$, and then clearly $\overline{\overline{T}} = \overline{T}$. 

Finally suppose that $T_1, T_2 \subseteq \spec_\mathcal{S}(R)$. Since $\ker(T_1 \cup T_2) = \ker(T_1) \cap \ker(T_2)$ we have that
\begin{align*}
\overline{T_1 \cup T_2}
		&= \setof{ \mathfrak{p} \in \spec_\mathcal{S}(R) }{ \mathfrak{p} \supseteq \ker(T_1 \cup T_2)} \\
		&= \setof{ \mathfrak{p} \in \spec_\mathcal{S}(R) }{ \mathfrak{p} \supseteq \ker(T_1) \cap \ker(T_2)} \\
		&= \setof{ \mathfrak{p} \in \spec_\mathcal{S}(R) }{ \mathfrak{p} \supseteq \ker(T_1) \text{ or } \mathfrak{p} \supseteq \ker(T_2)} \\
		&= \overline{T_1} \cup \overline{T_2}. 
\end{align*}
So \ref{ei:closure-union}. holds.
\end{proof}

We now describe the open sets in the Jacobson topology. 
To this end we define for each $I \in \mathcal{S}$ the set 
\[
 W(I) = \setof{ \mathfrak{p} \in \spec_\mathcal{S}(R) }{ \mathfrak{p} \nsupseteq I }.
\]

\begin{lemma} \label{l:opensets}
Let $R$ be a ring and let $\mathcal{S}$ be a sublattice of $\mathbb{I}(R)$ closed under arbitrary intersections and containing the trivial ideals $\{ 0 \}$ and $R$.  Then for all $U \subseteq \spec_\mathcal{S}(R)$, $U$ is open if and only if 
\[
  U = W(\ker(U^c)).
\]
Furthermore, if $I \in \mathcal{S}$ is such that 
\[
  I = \ker(\setof{\mathfrak{p} \in \spec_\mathcal{S}(R)}{\mathfrak{p} \supseteq I}),
\]
then $W(I)$ is open. 
\end{lemma}
\begin{proof}
Let $U$ be a subset of $\spec_\mathcal{S}(R)$.  Then $U$ is open if and only if $U^c = \overline{U^c}$ if and only if  
\[
  U^c = \setof{ \mathfrak{p} \in \spec_\mathcal{S}(R) }{ \mathfrak{p} \supseteq \ker(U^c) }
\]
if and only if 
\[
  U = \setof{ \mathfrak{p} \in \spec_\mathcal{S}(R) }{ \mathfrak{p} \nsupseteq \ker(U^c) } = W(\ker(U^c)). 
\]
%Since $U$ is open we have that $U^c = \overline{U^c}$, so 
%\begin{align*}
%  U^c = \overline{U^c} = \setof{ \mathfrak{p} \in \spec_\mathcal{S}(R) }{ \mathfrak{p} \supseteq \ker(U^c) }.
%\end{align*}
%Hence
%\begin{align*}
%  U = \setof{ \mathfrak{p} \in \spec_\mathcal{S}(R) }{ \mathfrak{p} \nsupseteq \ker(U^c) } = W(\ker(U^c)). 
%\end{align*}
Let now $I \in \mathcal{S}$ be such that 
\[
  I = \ker(\setof{\mathfrak{p} \in \spec_\mathcal{S}(R)}{\mathfrak{p} \supseteq I}).
\]
To ease notation we let $H = \setof{\mathfrak{p} \in \spec_\mathcal{S}(R)}{\mathfrak{p} \supseteq I}$, so that $I = \ker(H)$.
Then 
\begin{align*}
  W(I)^c	&= \setof{ \mathfrak{p} \in \spec_\mathcal{S}(R) }{ \mathfrak{p} \nsupseteq I }^c
		= \setof{ \mathfrak{p} \in \spec_\mathcal{S}(R)}{\mathfrak{p} \nsupseteq \ker(H)}^c \\
		&= \setof{ \mathfrak{p} \in \spec_\mathcal{S}(R)}{\mathfrak{p} \supseteq \ker(H)}
		= \overline{H}. 
\end{align*}
Hence $W(I)$ is open. 
\end{proof}

We now define a lattice isomorphism between the open sets of $\spec_\mathcal{S}(R)$ and the elements of $\mathcal{S}$. 

\begin{theorem}\label{t:latticeisofullring}
Let $R$ be a ring and let $\mathcal{S}$ be a sublattice of $\mathbb{I}(R)$ closed under arbitrary intersections and containing the trivial ideals $\{ 0 \}$ and $R$.
Suppose that for each $I \in \mathcal{S}$ we have that 
\begin{align*}
  I = \ker(\setof{\mathfrak{p} \in \spec_\mathcal{S}(R)}{\mathfrak{p} \supseteq I}).
\end{align*}
Define $\ftn{ \phi }{ \mathbb{O} ( \spec_\mathcal{S} ( R ) ) }{ \mathcal{S} }$ by 
\begin{align*}
\phi ( U ) = \ker(U^c).
\end{align*}
Then $\phi$ is a lattice isomorphism.
\end{theorem}

\begin{proof}
To show that $\phi$ is bijective we define $\ftn{ \gamma }{ \mathcal{S} } { \mathbb{O} ( \spec_\mathcal{S} ( R ) ) }$ by $\gamma ( I ) = W(I)$ and check that it is an inverse. 
Note that by Lemma \ref{l:opensets} the set $W(I)$ is in fact open.
For each $I \in \mathcal{S}$ we have  
\begin{align*}
  \phi(\gamma(I)) &= \phi(W(I)) = \ker(W(I)^c) = \ker(\setof{\mathfrak{p} \in \spec_\mathcal{S}( R )}{\mathfrak{p} \supseteq I}) = I,
\end{align*}
by the assumption on $I$. 
On the other hand, if $U \subseteq \spec_\mathcal{S}( R )$ is open we can use Lemma \ref{l:opensets} to get
\begin{align*}
  \gamma(\phi(U)) = \gamma(\ker(U^c)) = W(\ker(U^c)) = U. 
\end{align*}
Hence $\phi$ is bijective.
To show that $\phi$ is a lattice isomorphism it only remains to verify that both $\phi$ and $\gamma$ preserves order. 
Let $U,V$ be open subsets of $\spec_\mathcal{S}(R)$ with $U \subseteq V$.
Then $V^c \subseteq U^c$ so 
\begin{align*}
  \phi(U) = \ker(U^c) \subseteq \ker(V^c) = \phi(V),
\end{align*}
and hence $\phi$ is order preserving. 
Let now $I,J \in \mathcal{S}$ be such that $I \subseteq J$.  
Then 
\begin{align*}
  W(I)^c	= \setof{\mathfrak{p} \in \spec_\mathcal{S}(R)}{\mathfrak{p} \supseteq I} 
		\supseteq  \setof{\mathfrak{p} \in \spec_\mathcal{S}(R)}{\mathfrak{p} \supseteq J}
		= W(J)^c,
\end{align*}
which implies that $\gamma(I) = W(I) \subseteq W(J) = \gamma(J)$, i.e., $\gamma$ is order preserving. 
\end{proof}

In keeping with the notation from $C^*$-algebras we define $R[U] = \phi(U)$ for every $U \in \mathbb{O}(\prim_\mathcal{S}(R) )$. 
Whenever we have open sets $V \subseteq U$ we can form the quotient $R[U] / R[V]$. 
The next lemma shows that the quotient $R[U] / R[V]$ only depends on the set difference $U \setminus V$ up to canonical isomorphism. 

\begin{lemma}\label{l:subquot}
Let $R$ be a ring and let $\mathcal{S}$ be a sublattice of $\mathbb{I}(R)$ closed under arbitrary intersections and containing the trivial ideals $\{ 0 \}$ and $R$.
Suppose that for each $I \in \mathcal{S}$ we have that 
\begin{align*}
  I = \ker(\setof{\mathfrak{p} \in \spec_\mathcal{S}(R)}{\mathfrak{p} \supseteq I}).
\end{align*}
Then for all $U,V \in \mathbb{O}(\spec_\mathcal{S}(R) )$ we have
\begin{align*}
R [ U \cup V ] = R [ U ] + R [ V ] \quad \text{and} \quad R [ U \cap V ] = R [ U ] \cap R [ V ].
\end{align*}
Consequently, if $V_{1}, V_{2} , U_{1} , U_{2} \in \mathbb{O}( \spec_\mathcal{S}(R) )$  are such that $V_1 \subseteq U_1$, $V_2 \subseteq U_2$, and $U_{1} \setminus V_{1} = U_{2} \setminus V_{2}$, then there exits an isomorphism from $R [ U_{1} ] / R[ V_{1} ]$ to $R [ U_{2} ] / R [ V_{2} ]$ and this isomorphism is natural, i.e., if also $V_3,U_3\in \mathbb{O}(\spec_\mathcal{S}(R) )$ with $V_3\subseteq U_3$ and $U_3\setminus V_3=U_1 \setminus V_1$, then the composition of the isomorphisms from $R[U_1]/R[V_1]$ to $R[U_2]/R[V_2]$ and from $R[U_2]/R[V_2]$ to $R[U_3]/R[V_3]$ is equal to the isomorphism from $R[U_1]/R[V_1]$ to $R[U_3]/R[V_3]$.
\end{lemma}
\begin{proof}
The first part of the theorem follows from the fact that $\phi$ is a lattice isomorphism (Theorem \ref{t:latticeisofullring}) and that $\mathcal{S}$ is a sublattice.  

Suppose now $V_{1}, V_{2} , U_{1} , U_{2} \in \mathbb{O}(X)$ are as in the statement of the Lemma. 
Then $V_{1} \cup U_{2} = U_{1} \cup U_{2} = U_{1} \cup V_{2}$ and therefore
\begin{align*}
R [ U_{2} ] + R [ V_{1} ] = R [ V_{1} \cup U_{2} ] = R [ U_{1} \cup V_{2} ] = R [ U_{1} ] + R [ V_{2} ].
\end{align*}
Since $U_2 \cap (V_1 \cup V_2) = V_2$ we get
\begin{align*}
( R [ U_{2} ] + R [ V_{1} ]  ) / ( R [ V_{1} ] + R[V_{2}] ) &\cong R [ U_{2} ] / (R [ U_{2} ] \cap R[ V_{1} \cup V_{2} ])  \\
	&= R [ U_{2} ] / R [ U_{2} \cap ( V_{1} \cup V_{2} ) ] \\
	&= R [ U_{2} ] / R [ V_{2} ].
\end{align*}
Similarly
\begin{align*}
( R [ U_{1} ] + R [ V_{2} ]  ) / ( R [ V_{1} ] + R[V_{2}] ) \cong R[ U_{1} ] / R[ V_{1} ].
\end{align*}
Hence 
\begin{align*}
R[ U_{1} ] / R[ V_{1} ] &\cong ( R [ U_{1} ] + R [ V_{2} ]  ) / ( R [ V_{1} ] + R[V_{2}] ) \\
			&= ( R [ U_{2} ] + R [ V_{1} ]  ) / ( R [ V_{1} ] + R[V_{2}] ) \\
			& \cong R [ U_{2} ] / R [ V_{2} ].
\end{align*}

Suppose that we also have $V_3,U_3\in \mathbb{O}(\spec_\mathcal{S}(R) )$ with $V_3\subseteq U_3$ and $U_3\setminus V_3=U_1 \setminus V_1$.  Then 
\begin{align*}
V_1 \cup U_2 = U_1 \cup U_2 = U_1 \cup V_2, \\
V_2 \cup U_3 = U_2 \cup U_3 = U_2 \cup V_3, \\
V_1 \cup U_3 = U_1 \cup U_3 = U_1 \cup V_3, 
\end{align*} 
\begin{align*}
V_1 &= U_1 \cap ( V_1 \cup V_2 ) = U_1 \cap ( V_1 \cup V_3 ) = U_1 \cap ( V_1 \cup V_2 \cup V_3 ),  \\
V_2 &= U_2 \cap ( V_1 \cup V_2 ) = U_2 \cap ( V_2 \cup V_3 ) = U_2 \cap ( V_1 \cup V_2 \cup V_3 ),  \text{ and} \\
V_3 &= U_3 \cap ( V_1 \cup V_3 ) = U_3 \cap ( V_2 \cup V_3 ) = U_3 \cap ( V_1 \cup V_2 \cup V_3 ).
\end{align*}
Now, by considering the isomorphism constructed above, one then gets that the isomorphism is natural from Noether's isomorphism theorem. 
\end{proof}

\begin{definition}
Let $X$ be a topological space and let $Y$ be a subset of $X$. 
We call $Y$ \emph{locally closed} if $Y = U \setminus V$ where $U , V \in \mathbb{O} ( X )$ with $V \subseteq U$.
We let $\mathbb{LC}(X)$ be the set of locally closed subsets of $X$.
\end{definition}

\begin{definition}
Let $R$ be a ring and let $\mathcal{S}$ be a sublattice of $\mathbb{I}(R)$ closed under arbitrary intersections and containing the trivial ideals $\{ 0 \}$ and $R$.
Suppose that for each $I \in \mathcal{S}$ we have that 
\begin{align*}
  I = \ker(\setof{\mathfrak{p} \in \spec_\mathcal{S}(R)}{\mathfrak{p} \supseteq I}).
\end{align*}
For $Y = U \setminus V \in \mathbb{LC}( \spec_\mathcal{S}(R) )$, define
\begin{align*}
R [ Y ] := R [ U ] / R [ V ].
\end{align*}
By Lemma~\ref{l:subquot}, $R[Y]$ does not depend on $U$ and $V$ up to a canonical choice of isomorphism.
\end{definition}

%%%%%%%%%%%%%%%%%%%%%%%%%
%%%%%%%%%%%%%%%%%%%%%%%%%
\section{$\spec_\gamma ( L_\mathsf{k} (E) )$ and $\prim_\gamma ( C^*(E) )$}
%%%%%%%%%%%%%%%%%%%%%%%%%
%%%%%%%%%%%%%%%%%%%%%%%%%

Having set up our notion of primitive ideal spectrum relative to a sublattice, we will now apply it to the graded ideals of Leavitt path algebras. 

\begin{definition}
Let $E$ be a graph and let $\mathsf{k}$ be a field.
We denote by $\mathbb{I}_\gamma ( L_\mathsf{k}(E) )$ the sublattice of $\mathbb{I}( L_\mathsf{k}(E) )$ consisting of all graded ideals of $L_\mathsf{k}(E)$ and for brevity we let $\spec_\gamma ( L_\mathsf{k} (E) ) = \spec_{\mathbb{I}_\gamma(L_\mathsf{k}(E))}(L_\mathsf{k}(E))$.

Similarly we let $\mathbb{I}_\gamma ( C^*(E) )$ be the sublattice of $\mathbb{I}( C^*(E) )$ consisting of all gauge invariant ideals of $C^*(E)$.
\end{definition}

Recall from \cite[Section 3]{arXiv:1604.05439v1} that  $\prim_\gamma ( C^{*} (E) )$ denotes the collection of prime gauge invariant ideals of $C^*(E)$. 
We first prove that the lattice of graded ideals and the lattice of gauge invariant ideals are isomorphic in a canonical way. 

\begin{lemma} \label{l:isogradedgauge}
Let $E$ be a graph. 
The map $\ftn{ \beta }{ \mathbb{I}_\gamma ( L_\mathsf{k}(E) ) }{ \mathbb{I}_\gamma ( C^*(E) ) }$ that is given by $\beta(I^{\alg}_{(H,S)}) = I^{\tops}_{(H,S)}$ is a lattice isomorphism. 
Furthermore $\beta$ maps $\spec_\gamma ( L_\mathsf{k} (E) )$ bijectively onto $\prim_\gamma ( C^{*} (E) )$.
\end{lemma}
\begin{proof}
By Theorem~\ref{t:admissiblelattice} there is a lattice isomorphism $\beta_{\alg}$ from the set of admissible pairs to $\mathbb{I}_\gamma ( L_\mathsf{k}(E))$ given by $\beta_{\alg}((H,S)) = I^{\alg}_{(H,S)}$, and a lattice isomorphism $\beta_{\tops}$ from the set of admissible pairs to $\mathbb{I}_\gamma ( C^*(E))$ given by $\beta_{\tops}((H,S)) = I^{\tops}_{(H,S)}$. 
Consequently, $\beta = \beta_{\tops} \circ \beta_{\alg}^{-1}$ is a lattice isomorphism. 

Let $\mathcal{S} = \mathbb{I}_\gamma ( L_\mathsf{k}(E) )$. 
It follows from \cite[Proposition~II.1.4]{MR676974} that a graded ideal $I$ of $L_\mathsf{k}(E)$ is $\mathcal{S}$-prime if and only if $I$ is a prime ideal of $L_\mathsf{k} (E)$.
Thus, by \cite[Theorem~3.12]{MR2998948}, every $\mathcal{S}$-prime ideal $I$ of $L_{\mathsf{k}} (E)$ is of the form 
\begin{itemize}
  \item $I = I_{(H,S)}^{\alg}$, where $E^{0} \setminus H$ is a maximal tail and $S = B_{H}$, or
  \item $I = I_{ ( H , S )}^{\alg}$ where $E^{0} \setminus H = M( u )$ and $S = B_{ H} \setminus \{ u \}$ for some breaking vertex,
\end{itemize} 
and that these ideals are distinct. 
In \cite[Section~3]{arXiv:1604.05439v1} it is shown that every ideal $\mathfrak{I}$ in $\prim_\gamma ( C^{*} (E) )$ is of the form
\begin{itemize}
  \item $\mathfrak{I} = I_{(H,S)}^{\tops}$, where $E^{0} \setminus H$ is a maximal tail and $S = B_{H}$, or
  \item $\mathfrak{I} = I_{ ( H , S )}^{\tops}$ where $E^{0} \setminus H = M( u )$ and $S = B_{ H} \setminus \{ u \}$ for some breaking vertex,
\end{itemize} 
and that these ideals are distinct.
Hence $I_{(H,S)}^{\tops}$ is in $\prim_\gamma ( C^{*} (E) )$ if and only if $I_{(H,S)}^{\alg}$ is in $\spec_\gamma ( L_\mathsf{k} (E) )$.
In other words $\beta$ maps $\spec_\gamma ( L_\mathsf{k} (E) )$ bijectively onto $\prim_\gamma ( C^{*} (E) )$.
\end{proof}

We can now prove that the collection of graded ideals satisfies the kernel assumption we used in Section~\ref{s:primespectrum}. 

\begin{proposition} \label{p:intersections}
Let $E$ be a graph.
If $I$ is a proper graded ideal of $L_{\mathsf{k}} (E)$, then  
\begin{align*}
  I = \ker\left( \setof{\mathfrak{p} \in \spec_\gamma ( L_{\mathsf{k}} (E) )}{\mathfrak{p} \supseteq I } \right).
\end{align*}
\end{proposition}
\begin{proof}
Let $\beta$ be the lattice isomorphism from Lemma \ref{l:isogradedgauge} and let $I \in \mathbb{I}_\gamma ( L_\mathsf{k}(E))$ be a proper ideal. 

By \cite[Lemma~3.5]{arXiv:1604.05439v1} we have that 
\begin{align*}
  \beta(I) = \bigcap_{ \substack{ \mathfrak{q} \in \prim_\gamma ( C^{*} (E) ) \\ \mathfrak{q} \supseteq  \beta(I) }} \mathfrak{q}.
\end{align*}
Since $I$ is a graded ideal $I = I_{ (H,S)}^{\alg}$ for some admissible pair $(H,S)$.
As the intersection of graded ideals is again graded we also have
\begin{align*}
  \bigcap_{ \substack{ \mathfrak{p} \in \spec_\gamma ( L_{\mathsf{k}} (E) ) \\ \mathfrak{p} \supseteq I }} \mathfrak{p} = I_{ (H', S' )}^{\alg},
\end{align*}  
for some admissible pair $(H',S')$. 
We will now show that $I_{ (H,S)}^{\tops} = I_{ (H',S')}^{\tops}$.

Since $I_{ (H', S' )}^{\alg}$ is an intersection of ideals that all contain $I_{ (H,S)}^{\alg}$, $I_{ (H,S)}^{\alg} \subseteq I_{ (H',S')}^{\alg}$ which implies that $I_{ (H,S)}^{\tops} \subseteq I_{ (H',S')}^{\tops}$ as $\beta$ is order preserving. 
If $\mathfrak{q} \in \prim_\gamma ( C^{*} (E) )$ is such that $I_{ ( H , S )}^{\tops} \subseteq \mathfrak{q}$, then $I_{ (H,S)}^{\alg} \subseteq \beta^{-1}( \mathfrak{q} )$.
Therefore $\beta^{-1}(\mathfrak{q})$ is one the ideals whose intersection define $I_{ (H', S' )}^{\alg}$ so
\begin{align*}
  I_{ (H',S')}^{\tops} = \beta ( I_{ (H', S' )}^{\alg} ) \subseteq  \beta(  \beta^{-1} ( \mathfrak{q} ) ) = \mathfrak{q}.
\end{align*}
We now have the following inclusions 
\begin{align*}
  I_{ (H,S)}^{\tops}	\subseteq I_{ (H',S')}^{\tops} \subseteq \bigcap_{ \substack{ \mathfrak{q} \in \prim_\gamma ( C^{*} (E) )  \\ \mathfrak{q} \supseteq I_{ (H,S)}^{\tops} }  } \mathfrak{q}
			=  \beta(I) 
			=  I_{ (H,S)}^{\tops}.
\end{align*} 
Therefore, $I_{ (H,S)}^{\tops} = I_{ (H',S')}^{\tops}$. 
Hence $(H , S ) = (H', S' )$ so
\begin{align*}
  I = I_{(H,S)}^{\alg} = I_{ (H',S')}^{\alg} = \bigcap_{ \substack{ \mathfrak{p} \in \spec_\gamma ( L_{\mathsf{k}} (E) ) \\ \mathfrak{p} \supseteq I }} \mathfrak{p} = \ker\left( \setof{\mathfrak{p} \in \spec_\gamma ( L_{\mathsf{k}} (E) )}{\mathfrak{p} \supseteq I } \right)
\end{align*}
\end{proof}
%%%%%%%%%%%%%% I never wrote this:
%As $\beta$ is order preserving it gives a bijection between 
%\begin{align*}
%  \setof{ \mathfrak{p} \in \spec_\gamma( L_{\mathsf{k}} (E) ) }{ \mathfrak{p} \supseteq I_{(H,S)}^{\alg} }
%\end{align*}
%and 
%\begin{align*}
%  \setof{ \mathfrak{q} \in \prim_\gamma( C^{*} (E) ) }{ \mathfrak{q} \supseteq I_{(H,S)}^{\tops} }.
%\end{align*}
%%%%%%%%%%%%% Seems off. 

\begin{lemma}\label{lem:intersection-graded}
Let $\mathsf{k}$ be a field and let $\{ I_\alpha \}_{ \alpha \in S }$ be a subset of $\mathbb{I}_\gamma ( L_\mathsf{k}(E) )$.  Then the ideal $I = \bigcap_{ \alpha \in S } I_{\alpha}$ is an element of $\mathbb{I}_\gamma ( L_\mathsf{k}(E) )$.
\end{lemma}

\begin{proof}
Recall that $J \in \mathbb{I}_\gamma ( L_\mathsf{k}(E) )$ if and only if $J = \bigcap_{ n \in \Z }  J \cap L_{\mathsf{k}}(E)_n$,
where 
\[
L_{\mathsf{k}}(E)_n = \left\{  \sum_{k=1}^l \lambda_k \mu_k \nu_k^* : \text{$\lambda_k \in \mathsf{k}$, $\mu_k, \nu_k$ are finite paths, and $|\mu_k | - | \nu_k | = n$}\right\}.
\]
Then 
\begin{align*}
\bigcap_{ n \in \Z } I \cap L_\mathsf{k}(E)_n &= \bigcap_{ n \in \Z } \bigcap_{\alpha \in S } I_{\alpha} \cap L_\mathsf{k}(E)_n = \bigcap_{ \alpha \in S } \bigcap_{n \in \Z } I_\alpha \cap L_\mathsf{k}(E)_n \\
							&= \bigcap_{ \alpha \in S }  I_\alpha = I.
\end{align*}
Hence, $I \in \mathbb{I}_\gamma ( L_\mathsf{k}(E) )$.
\end{proof}

\begin{corollary}
The map 
\begin{align*}
U \mapsto \bigcap_{ \mathfrak{p} \in \spec_\gamma(R) \setminus U } \mathfrak{p}
\end{align*}
is a lattice isomorphism from $\mathbb{O} ( \spec_\gamma( L_{\mathsf{k}} (E) ) )$ to $\mathbb{I}_\gamma ( L_{\mathsf{k}} (E) )$.
\end{corollary}
\begin{proof}
This follows from Theorem~\ref{t:latticeisofullring} which is applicable by Proposition \ref{p:intersections} and Lemma~\ref{lem:intersection-graded}. 
\end{proof}

As the final result in this section we prove that $\beta$ restricts to a homeomorphism between the graded prime ideals and the gauge prime ideals. 

\begin{theorem}\label{t:specprim}
Let $E$ be a graph.  Then $\phi = \beta \vert_{\spec_\gamma( L_{\mathsf{k}} (E) ) }$ is a homeomorphism from $\spec_\gamma( L_{\mathsf{k}} (E) )$ to $\prim_\gamma ( C^{*} (E) )$, where $\beta$ is the lattice isomorphism from Lemma~\ref{l:isogradedgauge}.
\end{theorem}
        
\begin{proof} 
We first observe that Lemma~\ref{l:opensets} and Proposition~\ref{p:intersections} combine to show that the open sets of $\spec_\gamma( L_{\mathsf{k}} (E) )$ are precisely the sets of the form $W(I)$ for some proper ideal $I \in \mathbb{I}_\gamma ( L_\mathsf{k}(E) )$. 

Let a proper ideal $I \in \mathbb{I}_\gamma ( L_\mathsf{k}(E) )$ be given. 
Then 
\begin{align*}
  \beta(W(I))	&= \beta \left( \setof{\mathfrak{p} \in \spec_\gamma( L_{\mathsf{k}} (E) )}{ \mathfrak{p} \nsupseteq I} \right) \\
		&= \setof{ \beta(\mathfrak{p})}{ \mathfrak{p} \in \spec_\gamma( L_{\mathsf{k}} (E) ) \text{ and } \mathfrak{p} \nsupseteq I} \\
		&= \setof{ \beta(\mathfrak{p})}{ \mathfrak{p} \in \spec_\gamma( L_{\mathsf{k}} (E) ) \text{ and } \beta(\mathfrak{p}) \nsupseteq \beta(I)} \\
		&= \setof{ \mathfrak{q} \in \prim_\gamma ( C^{*} (E) ) }{ \mathfrak{q} \nsupseteq \beta(I)}.
\end{align*}
By \cite[Lemma 3.6]{arXiv:1604.05439v1} the last set is open, and hence $\phi^{-1}$ is continuous. 

The above computation used that $\beta$ was a lattice isomorphism and that we had complete, and similar looking, descriptions of the open sets in $\spec_\gamma( L_{\mathsf{k}} (E) )$ and $\prim_\gamma ( C^{*} (E) )$. 
Hence a completely parallel computation will show that $\phi$ is also continuous. 
Therefore $\phi$ is a homeomorphism. 
\end{proof}

\section{Filtered algebraic $K$-theory}

In this section we define filtered algebraic $K$-theory for rings and show that if two Leavitt path algebras over $\C$ have isomorphic filtered algebraic $K$-theory then the associated graph $C^*$-algebras have isomorphic filtered $K$-theory.   We then use this result to answer the Abrams-Tomforde conjecture for a large class of finite graphs.

Suppose $R$ is a ring and $\mathcal{S}$ is a sublattice of ideals.  Moreover, assume that every $I \in \mathcal{S}$ has a countable approximate unit consisting of idempotents, i.e., for every $I \in \mathcal{S}$, there exists a sequence $\{ e_n \}_{n =1}^\infty$ in $I$ such that 
\begin{itemize}
\item $e_n$ is an idempotent for all $n \in \N$,
\item $e_n e_{n+1} = e_n$ for all $n \in \N$, and 
\item for all $r \in I$, there exists $n \in \N$ such that $r e_n = e_n r = r$.
\end{itemize}  
Then for any locally closed subset $Y = U \setminus V$ of $\spec_\mathcal{S}(R)$, we have the algebraic $K$-groups $\{ K_{n}^{\alg} ( R [Y] ) \}_{ n \in \Z}$. 
Moreover, for all $U_1, U_2, U_3 \in \mathbb{O}(\spec_\mathcal{S}(R))$ with $U_1 \subseteq U_2 \subseteq U_3$, by \cite[Lemma~3.10]{MR3188556}, we have a long exact sequence in algebraic $K$-theory
\begin{align*}
\scalebox{.9}{
\xymatrix{
K_{n}^{\alg} ( R [ U_2\setminus U_1  ] ) \ar[r]^{ \iota_{*} } & K_{n}^{\alg} ( R [ U_3 \setminus U_1 ] ) \ar[r]^{ \pi_{*} } & K_{n}^{\alg} ( R [ U_3 \setminus U_2 ] ) \ar[r]^{ \partial_{*} } & K_{n-1}^{\alg} ( R [ U_2 \setminus U_1 ] ).
}
}
\end{align*}

\begin{definition}
Let $R$ be a ring and let $\mathcal{S}$ be a sublattice of $\mathbb{I}(R)$ closed under arbitrary intersections and containing the trivial ideals $\{ 0 \}$ and $R$.
Suppose that for each $I \in \mathcal{S}$ we have that 
\begin{align*}
  I = \ker(\setof{\mathfrak{p} \in \spec_\mathcal{S}(R)}{\mathfrak{p} \supseteq I}).
\end{align*}
Moreover, assume that every $I \in \mathcal{S}$ has a countable approximate unit consisting of idempotents.
\begin{enumerate}
  \item We define $\mathsf{FK}^{\alg} ( \spec_\mathcal{S} (R) ; R)$ to be the collection 
  \[
  \{ K_{n}^{\alg} ( R [ Y ] ) \}_{ n \in \Z , Y \in \mathbb{LC} (\prim_\mathcal{S}(R)) },
  \] equipped with the natural transformations $\{ \iota_{*} , \pi_{*} , \partial_{*} \}$.
  \item We define $\mathsf{FK}^{\alg, +} ( \spec_\mathcal{S} (R) , R)$ to be the collection $\mathsf{FK}^{\alg} ( \spec_\mathcal{S} (R) ; R)$ together with the positive cone of $K_{0}^{\alg} ( R [ Y ] )$ for all $Y \in \mathbb{LC} ( \spec_\mathcal{S}(R) )$.
\end{enumerate}
\end{definition}

\begin{definition}
Let $R, R'$ be rings, let $\mathcal{S}$ be a sublattice of $\mathbb{I}(R)$ closed under arbitrary intersections and containing the trivial ideals $\{ 0 \}$ and $R$, and let $\mathcal{S}'$ be a sublattice of $\mathbb{I}(R')$ closed under arbitrary intersections and containing the trivial ideals $\{ 0 \}$ and $R'$.
Suppose that for each $I \in \mathcal{S}$ we have that 
\begin{align*}
  I = \ker(\setof{\mathfrak{p} \in \spec_\mathcal{S}(R)}{\mathfrak{p} \supseteq I}),
\end{align*}
and that for each $I' \in \mathcal{S}'$ we have that 
\begin{align*}
  I' = \ker(\setof{\mathfrak{p} \in \spec_{\mathcal{S}'}(R')}{\mathfrak{p} \supseteq I'}).
\end{align*}
Moreover, assume that every $I \in \mathcal{S}$ and every $I' \in \mathcal{S}'$ have a countable approximate unit consisting of idempotents.  

An isomorphism from $\mathsf{FK}^{\alg} ( \spec_\mathcal{S} (R) ; R)$ to $\mathsf{FK}^{\alg} ( \spec_{\mathcal{S}'} (R') ; R')$ consists of a homeomorphism $\phi \colon \spec_\mathcal{S} ( R) \rightarrow \spec_{\mathcal{S}'} (R')$ and isomorphisms $\alpha_{Y, * }$ from $K_* ( R[Y ] ) $ to $K_* ( R'[\phi(Y)] )$ for each $Y \in \mathbb{LC} ( \spec_\mathcal{S} ( R) )$ such the diagrams involving the natural transformations commute.

If the isomorphism from $\mathsf{FK}^{\alg} ( \spec_\mathcal{S} (R) ; R)$ to $\mathsf{FK}^{\alg} ( \spec_{\mathcal{S}'} (R') ; R')$ restricts to an order isomorphism on $K_0 ( R[Y] )$ for all $Y \in \mathbb{LC} ( \spec_\mathcal{S} ( R) )$, we write
\[
    \mathsf{FK}^{\alg, +} ( \spec_\mathcal{S} (R) ; R) \cong \mathsf{FK}^{\alg, +} ( \spec_{\mathcal{S}'} (R') ; R').
\]
\end{definition}

\begin{lemma}\label{lem:idempotents}
Let $E$ be a graph and let $\mathsf{k}$ be a field.  Then every graded-ideal of $L_\mathsf{k} (E)$ has a countable approximate unit consisting of idempotents.  Consequently, we may define $\mathsf{FK}^{\alg, +} ( \spec_\gamma ( L_\mathsf{k} (E) ) ; L_\mathsf{k} (E) )$.
\end{lemma}

\begin{proof}
Let $F$ be a graph and set $F^0 = \{ v_1, v_2, \dots \}$.  Then $\{ \sum_{ k =1 }^n v_k \}_{n = 1}^\infty$ is a countable approximate unit consisting of idempotents for $L_\mathsf{k}(F)$.  Thus, every Leavitt path algebra has a countable approximate unit consisting of idempotents.   The lemma now follows since by \cite[Corollary~6.2]{MR3254772} every graded-ideal of $L_\mathsf{k}(E)$ is isomorphic to a Leavitt path algebra.  
\end{proof}

\begin{lemma}\label{lem:opensets-ideals}
Let $E$ be a directed graph and let $\ftn{\phi}{\spec_{\gamma} (L_\C(E))}{ \spec_{\gamma}( C^*(E) ) }$ be the homeomorphism given in Theorem~\ref{t:specprim}.  
%%%%
% badbox change
%The for all $U \in \mathbb{O} ( \spec_\gamma (L_\C(E) ) )$, there exists an admissible pair $(H,S)$ such that $L_\C (E) [ U ] = I^\alg_{(H,S)}$ and $C^*(E)[\phi(U)] = I^\tops_{(H,S)}$.  
The for all $U \in \mathbb{O} ( \spec_\gamma (L_\C(E) ) )$, there exists an admissible pair $(H,S)$ such that $L_\C (E) [ U ] = I^\alg_{(H,S)}$ and $C^*(E)[\phi(U)]$ $= I^\tops_{(H,S)}$. 
\end{lemma}

\begin{proof}
This follows from the construction of $\phi$ in Theorem~\ref{t:specprim} as the restriction of the lattice isomorphism $\beta$ that sends $I^\alg_{(H,S)}$ to $I_{(H,S)}^\tops$.  
\end{proof}

Let $\mathfrak{A}$ be a $C^*$-algebra and let $A$ be a $*$-algebra.  Suppose $\iota_\mathfrak{A}$ is a $*$-homomorphism from $A$ to $\mathfrak{A}$.  Denote the composition
\[
\xymatrix{ K_n^\mathrm{alg} (A) \ar[r]^{ K_n ( \iota_\mathfrak{A} ) } & K_n^\mathrm{alg} ( \mathfrak{A} ) \ar[r] & K_n^\tops ( \mathfrak{A} )}
\]
by  $\gamma_{n, \mathfrak{A}}$, where $K_n^\tops ( \mathfrak{A} )$ is the (usual) topological $K$-theory of the $C^*$-algebra $\mathfrak{A}$.

\begin{theorem}\label{t:alg-top-K-theory}
Let $E$ be a directed graph and let 
\[
\ftn{\phi}{\spec_{\gamma} (L_\C(E))}{\spec_{\gamma}( C^*(E) )}
\] 
be the homeomorphism given in Theorem~\ref{t:specprim}.  For all $U_1, U_2, U_3 \in \mathbb{O} ( \spec_{\gamma} (L_\C(E)) )$ with $U_1\subseteq U_2 \subseteq U_3$, for all $n \in \Z$, the diagram 
\[
\scalebox{.725}{
$\xymatrix{
K_n^\alg( L_\C (E)[ U_2 \setminus U_1] ) \ar[r] \ar[d]_{\gamma_{n, C^*(E)[V_2 \setminus V_1] } }& K_{n}^\alg ( L_\C(E)[U_3\setminus U_1] ) \ar[r] \ar[d]_{\gamma_{n, C^*(E)[V_3 \setminus V_1] } } & K_{n}^\alg ( L_\C(E)[U_3 \setminus U_2] ) \ar[r]  \ar[d]_{\gamma_{n, C^*(E)[V_3 \setminus V_2] } } & K_{n-1}^\alg ( L_\C(E)[U_2 \setminus U_1] ) \ar[d]_{\gamma_{n-1, C^*(E)[V_2 \setminus V_1] } } \\
K_n^\tops( C^*(E)[ V_2 \setminus V_1  ] ) \ar[r] & K_{n}^\tops ( C^*(E)[ V_3\setminus V_1  ] ) \ar[r] & K_{n}^\tops ( C^*(E)[ V_3 \setminus V_2  ] ) \ar[r] & K_{n-1}^\tops ( C^*(E)[ V_2 \setminus V_1  ] )
}$
}
\]
is commutative, where $V_i = \phi ( U_i)$
\end{theorem}

\begin{proof}
This follows Lemma~\ref{lem:opensets-ideals} and from \cite[Theorems~2.4.1 and~3.1.9]{MR2762555} . 
\end{proof}

%Recall that every graded ideal of $L_\C (E)$ is of the form $I_{(H, S)}^\alg$, the ideal generated by 
%\[
%v \quad \text{for all $v \in H$} 
%\]
%and 
%\[
%v^H := v - \sum_{ \substack{ s(e) =v \\ r(e) \notin H } }  ee^* \quad \text{for all $v \in S$,}
%\]
%where $H$ is a hereditary and saturated subset of $E^0$ and $S$ is a subset of 
%\[
%B_H := \{ v \in E^0 : \text{$| s^{-1} (v)| = \infty$ and $| s^{-1}(v) \cap r^{-1} ( E^0 \setminus H) | < \infty$} \}. 
%\]
%Here $(H,S)$ is called an \emph{admissible pair}.
%
%Similarly, that every gauge-invariant ideal of $C^*(E)$ is of the form $I_{(H, S)}^\tops$, the ideal generated by 
%\[
%p_v \quad \text{for all $v \in H$} 
%\]
%and 
%\[
%p_v^H := p_v - \sum_{ \substack{ s(e) =v \\ r(e) \notin H } }  s_es_e^* \quad \text{for all $v \in S$,}
%\]
%where $H$ is a hereditary and saturated subset of $E^0$ and $S \subseteq B_H$. 
%
\begin{lemma}\label{lem:induce-map-Ktheory}
Let $E$ be a graph.  Then for all $(H_1, S_1) , (H_2, S_2)$ admissible pairs with $( H_1, S_1 ) \leq ( H_2, S_2 )$, we have that 
\[
\gamma_{0, I_{(H_2, S_2) }^\tops / I_{ ( H_1, S_1) }^\tops } \colon K_0^\alg ( I_{(H_2, S_2)}^\alg / I_{ (H_1, S_1)}^\alg ) \rightarrow K_0^\tops ( I_{( H_2, S_2) }^\tops/ I_{ ( H_1, S_1) }^\tops )
\]
is an order isomorphism and 
\[
\gamma_{1, I_{(H_2, S_2) }^\tops / I_{ ( H_1, S_1) }^\tops } \colon K_1^\alg ( I_{(H_2, S_2)}^\alg / I_{ (H_1, S_1)}^\alg ) \rightarrow K_1^\tops( \mathfrak{I}_{( H_2, S_2) }/ \mathfrak{I}_{ ( H_1, S_1) } )
\]
is surjective with kernel a divisible group.

Suppose $F$ is a graph and suppose there exists an order isomorphism 
\[
\alpha_0 \colon  K_0^\alg( I_{(H_2, S_2)}^\alg / I_{ (H_1, S_1)}^\alg ) \rightarrow K_0^\alg ( I_{(H_2', S_2')}^\alg / I_{ (H_1', S_1')}^\alg )
\]
and there exists an isomorphism  
\[
\alpha_1 \colon  K_1^\alg ( I_{(H_2, S_2)}^\alg / I_{ (H_1, S_1)}^\alg ) \rightarrow K_1^\alg ( I_{(H_2', S_2')}^\alg / I_{ (H_1', S_1')}^\alg ),
\]
where $(H_i, S_i )$ is an admissible pair of $E$ for $i = 1, 2$ and $( H_i' , S_i')$ is an admissible pair of $F$ for $i = 1,2$ with $( H_1, S_1) \leq (H_2, S_2)$ and $(H_1' , S_1' ) \leq (H_2' , S_2' )$.  Then $\alpha_0$ and $\alpha_1$ induce isomorphisms 
\[
\widetilde{\alpha}_0 \colon  K_0^\tops ( I_{(H_2, S_2)}^\tops / I_{ (H_1, S_1)}^\tops ) \rightarrow K_0^\tops( I_{(H_2', S_2')}^\tops / I_{ (H_1', S_1')}^\tops )
\]
and 
\[
\widetilde{\alpha}_1 \colon  K_1^\tops ( I_{(H_2, S_2)}^\tops/ I_{ (H_1, S_1)}^\tops) \rightarrow K_1^\tops ( I_{(H_2', S_2')}^\tops / I_{ (H_1', S_1')}^\tops )
\]
such that $\widetilde{\alpha}_0$ is an order isomorphism and 
\[
\gamma_{i , I_{(H_2', S_2')}^\tops / I_{ (H_1', S_1')}^\tops } \circ \alpha_i = \widetilde{\alpha}_i \circ \gamma_{i,  I_{(H_2, S_2)}^\tops / I_{ (H_1, S_1)}^\tops }.
\]
\end{lemma}

\begin{proof}
Let $\iota_E \colon L_\C (E) \rightarrow C^*(E)$ be the $*$-homomorphism sending $v$ to $p_v$ and $e$ to $s_e$.  Note that for all admissible pairs $(H, S)$, $\iota_E ( I_{(H,S)}^\alg ) \subseteq I_{ (H,S) }^\tops$.  Therefore, for all admissible pairs $(H_1, S_1) , (H_2, S_2)$ with $( H_1, S_1 ) \leq ( H_2, S_2 )$, $\iota_E$ induces a $*$-homomorphism from $I_{(H_2, S_2)}^\alg / I_{ (H_1, S_1)}^\alg$ to $I_{(H_2, S_2)}^\tops / I_{ (H_1, S_1)}^\tops$.  We denote this map by $\iota_{E, I_{(H_2, S_2)}^\tops / I_{ (H_1, S_1)}^\tops }$.  Thus, the composition of this induced map in $K$-theory with the homomorphism from $K_n^\mathrm{alg} (I_{(H_2, S_2)}^\tops / I_{ (H_1, S_1)}^\tops )$ to $K_n^\tops (I_{(H_2, S_2)}^\tops / I_{ (H_1, S_1)}^\tops )$ is $\gamma_{n, I_{(H_2, S_2) }^\tops / I_{(H_1, S_1)}^\tops}$. 

We will show that it is enough to prove the first part of the lemma for the case $(H_2, S_2 ) = ( \emptyset, \emptyset)$ and $( H_1 ,S_1) = ( E^0 , \emptyset)$.  Let $(H, S)$ be an admissible pair.  Let $\overline{E}_{ ( H, S)}$ be the graph given in \cite[Definition~4.1]{MR3254772}.  By the proofs of \cite[Theorems~5.1 and~6.1]{MR3254772}, there exist $*$-isomorphisms
\[
\beta_{ (H, S) } \colon L_\C ( \overline{E}_{( H, S) } ) \rightarrow I_{ ( H, S) }^\alg \quad \text{and} \quad \lambda_{ ( H, S) } \colon  C^* ( \overline{E}_{( H, S) } ) \rightarrow I_{ ( H, S) }^\tops 
\]
given by 
\begin{align*}
\beta_{ ( H, S) } ( v) &:= \begin{cases} 
v & \text{ if $v \in H$} \\
v^H & \text{ if $v \in S$} \\
\alpha \alpha^* & \text{ if $v = \alpha \in F_1(H,S)$} \\
\alpha r(\alpha)^H \alpha^*& \text{ if $v = \alpha \in F_2(H,S)$} \\
\end{cases} \\
\beta_{(H, S)} (e) &:= \begin{cases} 
e & \text{ if $e \in E^1$} \\
\alpha & \text{ if $e = \overline{\alpha} \in \overline{F}_1(H,S)$} \\
\alpha r(\alpha)^H & \text{ if $e = \overline{\alpha} \in \overline{F}_2(H,S)$} \\
\end{cases} \\
\beta_{(H,S)}(e^*) &:= \begin{cases} 
e^* & \text{ if $e \in E^1$} \\
\alpha^* & \text{ if $e = \overline{\alpha} \in \overline{F}_1(H,S)$} \\
{r(\alpha)}^H \alpha^*  & \text{ if $e = \overline{\alpha} \in \overline{F}_2(H,S)$} \\
\end{cases} 
\end{align*}
and 
\begin{align*}
\lambda_{ ( H, S) } ( q_v) &:= \begin{cases} 
p_v & \text{ if $v \in H$} \\
p_v^H & \text{ if $v \in S$} \\
s_\alpha s_\alpha^* & \text{ if $v = \alpha \in F_1(H,S)$} \\
s_\alpha p_{r(\alpha)}^H s_\alpha^*& \text{ if $v = \alpha \in F_2(H,S)$} \\
\end{cases}\\
 \lambda_{(H, S)} (t_e) &:= \begin{cases} 
s_e & \text{ if $e \in E^1$} \\
s_\alpha & \text{ if $e = \overline{\alpha} \in \overline{F}_1(H,S)$} \\
s_\alpha p_{r(\alpha)}^H & \text{ if $e = \overline{\alpha} \in \overline{F}_2(H,S)$}. \\
\end{cases}
\end{align*}
Note that the diagram
\[
\xymatrix{
L_\C ( \overline{E}_{ (H,S)} ) \ar[rr]^-{ \iota_{ \overline{E}_{ (H,S)}  } } \ar[d]_{\beta_{(H,S)}} & & C^* ( \overline{E}_{ (H,S)}  ) \ar[d]^{\lambda_{ (H,S)}} \\
I_{ (H,S) }^\alg \ar[rr]_{\iota_{E, I_{ (H,S) }^\tops / 0 }  } & & I_{(H,S)}^\tops
}
\]
commutes.  
%%%%
% badbox change
%Therefore, for admissible pairs $(H_1, S_1) , (H_2, S_2)$ with $( H_1, S_1 ) \leq ( H_2, S_2 )$, the diagram 
Therefore, for two admissible pairs $(H_1, S_1) , (H_2, S_2)$ with $( H_1, S_1 ) \leq ( H_2, S_2 )$, the diagram 
\[
\xymatrix{
L_\C ( \overline{E}_{ (H_2,S_2)} ) / \beta_{ (H_2, S_2)}^{-1} ( I_{(H_1,S_1)}^\alg ) \ar[rr]^-{ \overline{\iota}_{ \overline{E}_{ (H_2,S_2)} } } \ar[d]_{\overline{\beta}_{(H_2,S_2)}} & & C^* ( \overline{E}_{ (H_2,S_2)}  )/ \lambda_{ (H_2, S_2)}^{-1} ( I_{(H_1,S_1)}^\tops ) \ar[d]^{\overline{\lambda}_{ (H_2,S_2)}} \\
I_{ (H_2,S_2) }^\alg / I_{(H_1,S_1)}^\alg \ar[rr]_{\iota_{E, I_{ (H_2,S_2) }^\tops / I_{(H_1,S_1) }^\tops}  } & & I_{(H_2,S_2)}^\tops/ I_{(H_1,S_1)}^\tops
}
\]
where $\overline{\beta}_{(H_2,S_2)}$ and $\overline{\lambda}_{ (H_2,S_2)}$ are the induced $*$-isomorphisms on the quotient, commutes.  Therefore, it is enough to prove the lemma for the graph $\overline{E}_{( H_2, S_2) }$.  Hence, we may assume that $(H_2, S_2) = ( E^0, \emptyset)$. 

Set $(H_1, S_1) = (H,S)$ to simplify the notation.  Let $E \setminus (H,S)$ be the graph defined in  \cite[Theorem~5.7(2)]{MR2363133}.  Then by the proof of \cite[Theorem~5.7(2)]{MR2363133} and the discussion before \cite[Corollary~5.7]{MR1988256}, there are $*$-isomorphisms
\[
\delta_{ (H, S) } \colon L_\C ( E\setminus (H,S) ) \rightarrow L_\C (E) / I_{(H,S)} 
\]
and
\[
\eta_{ ( H, S) } \colon  C^* ( E\setminus (H,S) ) \rightarrow C^*(E) / \mathfrak{I}_{ ( H, S) } 
\]
such that the diagram
\[
\xymatrix{
L_{\C} ( E \setminus (H,S)  ) \ar[rr]^-{\delta_{(H,S)} } \ar[d]_{ \iota_{E \setminus (H,S) } } & & L_\C (E) / I_{(H,S)} \ar[d]^{ \iota_{E, C^*(E)/ \mathfrak{I}_{ (H,S)} } }\\
C^* ( E \setminus (H,S)  ) \ar[rr]_-{\eta_{(H,S)} } & & C^*(E) / \mathfrak{I}_{(H,S)} 
}
\]
commutes. 
Hence, it is enough to prove the lemma for the graph $E \setminus (H, S )$.  Hence, we may assume that $(H, S) = ( \emptyset, \emptyset)$.  Thus, proving the claim. 

The fact that $\gamma_{0, C^*(E) / 0 }$ is an isomorphism follows from \cite[Corollary~3.5]{MR3310950}.  To prove that $\gamma_{1, C^*(E) / 0 }$ is surjective and its kernel is a divisible group we reduce to the case that $E$ is row-finite.  Let $F$ be a Drinen-Tomforde desingularization of $E$ defined in \cite{MR2117597}.  Then there are embeddings $\omega \colon L_\C (E) \rightarrow L_\C (F)$ and $\rho \colon C^*(E) \rightarrow C^*(F)$ such that the diagram 
\[
\xymatrix{
L_\C(E) \ar[r]^{\omega} \ar[d]_{\iota_E} & L_\C(F) \ar[d]^{\iota_F} \\
C^*(E) \ar[r]_{\rho} & C^*(F) 
}
\]
commutes, $\omega( L_\C (E) )$ is a full corner of $L_\C(F)$, and $\rho ( C^*(E) )$ is a full corner of $C^*(F)$.  Hence, $\omega$ and $\rho$ induce isomorphisms in $K$-theory.  Therefore, it is enough to prove $\gamma_{1, C^*(E) , 0 }$ is surjective with kernel a divisible group for the case that $E$ is row-finite.  The row-finite case follows from \cite[Lemma~4.7]{MR3188556}.  The first part of the lemma now follows.

For the last part of the lemma, since $K_0 (\iota_{ E,  I_{(H_2, S_2)}^\tops / I_{ (H_1, S_1)}^\tops } )$ is an order isomorphism, it is clear that $\alpha_0$ induces an order isomorphism $\widetilde{\alpha}_0$ such that 
\[
\gamma_{0 , I_{(H_2', S_2')}^\tops / I_{ (H_1', S_1')}^\tops } \circ \alpha_0 = \widetilde{\alpha}_0 \circ \gamma_{0,  I_{(H_2, S_2)}^\tops / I_{ (H_1, S_1)}^\tops }
\]  
The fact that $\alpha_1$ induces an isomorphism $\widetilde{\alpha}_1$ such that $\gamma_{1, I_{(H_2', S_2')}^\tops / I_{ (H_1', S_1')}^\tops } \circ \alpha_1 = \widetilde{\alpha}_1 \circ \gamma_{1,  I_{(H_2, S_2)}^\tops / I_{ (H_1, S_1)}^\tops }$ is the result of the kernel of $\gamma_{1, I_{(H_2, S_2) }^\tops / I_{ ( H_1, S_1) }^\tops }$ being a divisible group and $K_1 ( I_{(H_2', S_2')}^\tops / I_{ (H_1', S_1')}^\tops )$ being torsion free, thus \cite[Lemma~4.8]{MR3188556} applies.
\end{proof}

\begin{theorem}\label{t:iso-filtered-kthy}
Let $E$ and $F$ be graphs.  
\begin{enumerate}
\item  Suppose $\mathsf{FK}^{\alg, +} ( \spec_\gamma (L_\C(E)) ; L_\C(E) ) \cong \mathsf{FK}^{\alg, +} ( \spec_\gamma (L_\C(F)) ; L_\C(F) )$.  Then $\mathsf{FK}^{\tops, +} ( \spec_\gamma (C^*(E)) ; C^*(E) ) \cong \mathsf{FK}^{\tops, +} ( \spec_\gamma (C^*(F)) ; C^*(F) )$.

\item  Suppose $| E^0 |, | F^0 | < \infty$.  If
\[
\theta \colon \mathsf{FK}^{\alg, +} ( \spec_\gamma (L_\C(E)) ; L_\C(E) )  \rightarrow \mathsf{FK}^{\alg, +} ( \spec_\gamma (L_\C(F)) ; L_\C(F) ) 
\]
is an isomorphism such that $\theta_0$ sends $[1_{L_\C(E)} ]_0 \in K_0^\alg( L_\C(E) )$ to $[1_{L_\C(F)} ]_0 \in K_0^\alg( L_\C(F) )$,
then there exists an isomorphism
\[
\Theta \colon \mathsf{FK}^{\tops, +} ( \prim_\gamma (C^*(E)) ; C^*(E) ) \rightarrow \mathsf{FK}^{\tops, +} ( \prim_\gamma (C^*(F)) ; C^*(F) ) 
\]
such that $\Theta_0$ sends $[1_{C^*(E)} ]_0 \in K_0^\tops( C^*(E) )$ to $[1_{C^*(F)} ]_0 \in K_0^\tops( C^*(F) )$.
\end{enumerate}
\end{theorem}

\begin{proof}
The theorem follows from Lemmas~\ref{lem:opensets-ideals} and~\ref{lem:induce-map-Ktheory}, and Theorem~\ref{t:alg-top-K-theory}.
\end{proof}

\begin{corollary} \label{c:maincor}
Let $E$ and $F$ be graphs.  
\begin{enumerate}
\item If $L_\C(E)$ and $L_\C (F)$ are isomorphic as rings, then 
\[
\mathsf{FK}^{\tops, +} ( \prim_\gamma (C^*(E)) ; C^*(E) ) \cong \mathsf{FK}^{\tops, +} ( \prim_\gamma (C^*(F)) ; C^*(F) ).
\]  
If, in addition, $|E^0| , |F^0| < \infty$, then there exists an isomorphism 
\[
\Theta \colon \mathsf{FK}^{\tops, +} ( \prim_\gamma (C^*(E)) ; C^*(E) ) \rightarrow \mathsf{FK}^{\tops, +} ( \prim_\gamma (C^*(F)) ; C^*(F) )
\]
such that $\Theta_0$ sends $[1_{C^*(E)} ]_0 \in K_0^\tops( C^*(E) )$ to $[1_{C^*(F)} ]_0 \in K_0^\tops( C^*(F) )$.

\item If $L_\C(E)$ and $L_\C (F)$ are Morita equivalent, then 
\[
\mathsf{FK}^{\tops, +} ( \prim_\gamma (C^*(E)) ; C^*(E) ) \cong \mathsf{FK}^{\tops, +} ( \prim_\gamma (C^*(F)) ; C^*(F) ).
\]
\end{enumerate}
\end{corollary}

\begin{proof}
1. follows from Lemma~\ref{l:gradeidempotents} and Theorem~\ref{t:iso-filtered-kthy}.

Suppose $L_\C(E)$ and $L_\C(F)$ are Morita equivalent.  Then by \cite[Corollary~9.11]{MR2775826}, $M_\infty ( L_\C(E) ) \cong M_\infty ( L_\C (F) )$ as rings.  By \cite[Proposition~9.8(2)]{MR2775826}, $M_\infty ( L_\C(E) ) \cong L_\C ( SE )$ and $M_\infty ( L_\C (F) ) \cong L_\C ( SF )$ as $\C$-algebras, where $SE$ and $SF$ are the stabilized graphs of $E$ and $F$ respectively (see \cite[Definition~9.4]{MR2775826}).  Note that every graded ideal $L_\C (SE)$ is of the from $M_\infty (I)$ for a unique graded ideal of $I$ of $L_\C (E)$ and every graded ideal of $L_\C (SF)$ is of the from $M_\infty (J)$ for a unique graded ideal $J$ of $L_\C (F)$.  We also have that 
\begin{align*}
\mathsf{FK}^{\alg, +} ( \spec_\gamma (L_\C(E)) ; L_\C(E) ) &\cong \mathsf{FK}^{\alg, +} ( \spec_\gamma (L_\C(SE)) ; L_\C(SE) ) \\
					&\cong \mathsf{FK}^{\alg, +} ( \spec_\gamma (L_\C(SF)) ; L_\C(SF) ) \\
					&\cong \mathsf{FK}^{\alg, +} ( \spec_\gamma (L_\C(F)) ; L_\C(F) ).
\end{align*}
Therefore, by Theorem~\ref{t:iso-filtered-kthy},
\[
\mathsf{FK}^{\tops, +} ( \prim_\gamma (C^*(E)) ; C^*(E) ) \cong \mathsf{FK}^{\tops, +} ( \prim_\gamma (C^*(F)) ; C^*(F) ).
\]
\end{proof}

\begin{corollary}
The Abrams-Tomforde conjecture holds for the class of finite graphs that satisfy Condition (H) of \cite[Definition 4.19]{arXiv:1604.05439v1}. 
In particular the Abrams-Tomforde conjecture holds for the class of finite graphs that satisfy Condition (K). 
\end{corollary}
\begin{proof}
The first part is just a combination of Corollary~\ref{c:maincor} and \cite[Theorem 6.1]{arXiv:1604.05439v1}.
Finally, all graphs that satisfy Condition (K) satisfy Condition (H).
\end{proof}

\section*{Acknowledgements}
%%%%%%%%%%%%%%%%%%%%%%%%%%%%%%%%%%%%%%%%%%%%%%%%%%%%%%%%%%
%%% SECTION: Acknowledgement %%%%%%%%%%%%%%%%%%%%%%%%%%%%%
%%%%%%%%%%%%%%%%%%%%%%%%%%%%%%%%%%%%%%%%%%%%%%%%%%%%%%%%%%

This work was partially supported by the Danish National Research Foundation through the Centre for Symmetry and Deformation (DNRF92), by the VILLUM FONDEN through the network for Experimental Mathematics in Number Theory, Operator Algebras, and Topology, by a grant from the Simons Foundation (\# 279369 to Efren Ruiz), and by the Danish Council for Independent Research | Natural Sciences. 

\newcommand{\etalchar}[1]{$^{#1}$}

\end{document}